\documentclass [reqno,12pt]{amsart}
\usepackage{amsfonts}
\usepackage{amsmath}
\usepackage{amssymb}
\usepackage{graphicx}
\newtheorem{theorem}{Theorem}[section]
\theoremstyle{plain}

\newtheorem{corollary}{Corollary}[section]

\newtheorem{definition}{Definition}[section]

\newtheorem{lemma}{Lemma}[section]

\newtheorem{proposition}{Proposition}[section]
\newtheorem{remark}{Remark}[section]

\numberwithin{equation}{section} \textheight  22 true cm \textwidth  15 true cm  
\begin{document}
\title[Conformal anti-invariant $\xi^\perp-$submersions]{%
CONFORMAL ANTI-INVARIANT $\xi^\perp-$SUBMERSIONS}
\author{Mehmet Akif Akyol}
\address{Bing\"{o}l University, Faculty of Arts and Sciences, Deparment of Mathematics, 12000, Bing\"{o}l,
Turkey}
\email{mehmetakifakyol@bingol.edu.tr}
\author{Y\i lmaz G\"{u}nd\"{u}zalp}
\address{Dicle University, Faculty of Arts and Sciences, Deparment of Mathematics, 21280, Diyarbak\i r, Turkey}
\email{ygunduzalp@dicle.edu.tr}
\subjclass[2010]{53C15, 53C40.}
\keywords{Almost contact metric manifold, conformal submersion, conformal anti-invariant $\xi^\perp-$submersion.}

\begin{abstract}
As a generalization of anti-invariant $\xi^\perp-$Riemannian submersions, we introduce conformal anti-invariant $\xi^\perp-$submersions from almost contact
metric manifolds onto Riemannian manifolds. We investigate the geometry of foliations which are arisen from the definition of a conformal submersion
and find necessary and sufficient conditions for a conformal anti-invariant $\xi^\perp-$submersion to be totally geodesic and harmonic, respectively.
Moreover, we show that there are certain product structures on the total space of a conformal anti-invariant $\xi^\perp-$submersion.
\end{abstract}

\maketitle

\section{Introduction}
Riemannian submersions between Riemannian manifolds were studied by O'Neill \cite{O} and Gray
\cite{Gray}, for recent developments on the geometry of Riemannian
submanifolds and Riemannian submersions, see:\cite{Chen} and \cite{FIP}, respectively. In \cite{watson}, the Riemannian submersions were considered between almost Hermitian manifolds by
Watson under the name of almost Hermitian submersions. In this case, the Riemannian submersion
is also an almost complex mapping and consequently the vertical and horizontal distribution are
invariant with respect to the almost complex structure of the total manifold of the submersion. The
study of anti-invariant Riemannian submersions from almost Hermitian manifolds were initiated by
\c{S}ahin \cite{s1}. In this case, the fibres are anti-invariant with respect to the almost complex structure
of the total manifold. Beside there are many notions related with anti-invariant
Riemannian submersion (see: \cite{Asa}, \cite{BEM}, \cite{CE}, \cite{GIP}, \cite{Gun1}, \cite{Gun2}, \cite{Gun3}, \cite{LPSS},
 \cite{Park1}, \cite{Park2}, \cite{Park3}, \cite{Park4}, \cite{Fatima-Shahid1}, \cite{s2}, \cite{s3}, \cite{Tastan1}).
In \cite{Chi4}, Chinea defined almost contact Riemannian submersions between
almost contact metric manifolds and examined the differential geometric properties of Riemannian
submersions between almost contact metric manifolds. More precisely, let $(M_1,\phi,\xi,\eta,g_1)$ and
$(M_2,\phi^{'},\xi^{'},\eta^{'},g_2)$ be almost contact manifolds with $dim M_1=2m+1$ and $dim M_2=2n+1$. A
Riemannian submersion $\pi:M_1\longrightarrow M_2$ is called the almost contact metric submersion if $\pi$ is an
almost contact mapping, i.e., $\phi^{'}\pi_*=\pi_*\phi$. An immediate consequence of the above definition is that the vertical and
horizontal distributions are $\phi$-invariant. Moreover, the characteristic vector field $\xi$ is horizontal.
We note that only $\phi$-holomorphic submersions have been considered on almost contact manifolds
\cite{Chi4}.

One the other hand, as a generalization of Riemannian submersion,
horizontally conformal submersions are defined as follows \cite{BW}: Suppose
that $(M,g_{M})$ and $(B,g_{B})$ are Riemannian manifolds and $\pi:M\longrightarrow B$ is a smooth submersion,
then $\pi$ is called a horizontally conformal submersion, if there is a positive function $\lambda$ such that
\begin{equation*}
\lambda^{2}g_{M}(X,Y)=g_{B}(\pi_{*}X,\pi_{*}Y)
\end{equation*}
for every $X,Y\in\Gamma((ker\pi_{*})^\perp).$ It is obvious that every
Riemannian submersion is a particular horizontally conformal submersion with
$\lambda=1$. We note that horizontally conformal submersions are special
horizontally conformal maps which were introduced independently by Fuglede
\cite{F} and Ishihara \cite{I}. We also note that a horizontally conformal
submersion $\pi:M\longrightarrow B$ is said to be horizontally homothetic if
the gradient of its dilation $\lambda$ is vertical, i.e.,
\begin{equation}
\mathcal{H}(grad\lambda)=0
\end{equation}
at $p\in M$, where $\mathcal{H}$ is the projection on the horizontal space
$(ker\pi_{*})^{\perp}$. For conformal submersion, see: \cite{BW}, \cite{G}, \cite{Ornea}.

As a generalization of holomorphic submersions, conformal holomorphic submersions were studied
by Gudmundsson and Wood \cite{GW}. They obtained necessary and sufficient conditions for con-
formal holomorphic submersions to be a harmonic morphism, see also \cite{Chi1}, \cite{Chi2} and \cite{Chi3} for the
harmonicity of conformal holomorphic submersions.

Recently, in \cite{As} we have introduced conformal anti-invariant submersions from almost Hermitian
manifolds onto Riemannian manifolds and investigated the geometry of such submersions. (See also:\cite{A})
We showed that the geometry of such submersions are different from anti-invariant Riemannian
submersions.  In this paper, we consider conformal anti-invariant $\xi^\perp-$submersions from an almost contact
metric manifold under the assumption that the fibers are anti-invariant with respect to the tensor field of
type $(1,1)$ of the almost contact manifold.

The paper is organized as follows. In the second section,
we gather main notions and formulas for other sections. In section 3, we introduce conformal
anti-invariant $\xi^\perp-$submersions from almost contact metric manifolds onto Riemannian manifolds,
investigates the geometry of leaves of the horizontal distribution and the vertical distribution and
find necessary and sufficient conditions for a conformal anti-invariant $\xi^\perp-$submersion to be totally
geodesic and harmonic, respectively. In section 4, we show that there are certain product structures
on the total space of a conformal anti-invariant $\xi^\perp-$submersion.

\section{Preliminaries}

In this section, we define almost contact metric manifolds, recall the notion of
(horizontally) conformal submersions between Riemannian manifolds and give a
brief review of basic facts of (horizontally) conformal submersions.

Let $(M,g_{M})$ be an almost contact metric manifold with structure tensors $(\phi,\xi,\eta,g_{M})$
where $\phi$ is a tensor field of type (1,1), $\xi$ is a vector field, $\eta$ is a 1-form and $g_{M}$ is the Riemannian metric on $M.$
Then these tensors satisfy \cite{B}
\begin{equation}
\phi\xi=0,\ \ \eta o\phi=0,\ \ \eta(\xi)=1 \label{e.q:2.1}
\end{equation}
\begin{equation}
\phi^{2}=-I+\eta\otimes\xi \ \ \ \text{and} \ \ g_{M}(\phi X,\phi Y)=g_{M}(X,Y)-\eta(X)\eta(Y),  \label{e.q:2.2}
\end{equation}
where $I$ denotes the identity endomorphism of $TM$ and $X,Y$ are any vector fields on $M$. Moreover, if $M$ is Sasakian \cite{SH}, then we have
\begin{equation}
(\nabla_{X}\phi)Y=-g_{M}(X,Y)\xi+\eta(Y)X \ \ \text{and} \ \ \nabla_{X}\xi=\phi X, \label{e.q:2.3}
\end{equation}
where $\nabla$ is the connection of Levi-Civita covariant differentiation.

Conformal submersions belong to a wide class of conformal maps that we are going to recall their
definition, but we will not study such maps in this paper.

\begin{definition}
\textit{(\cite{BW})} Let $\varphi:(M^{m},g)\longrightarrow (N^{n},h)$ be a
smooth map between Riemannian manifolds, and let $x\in M$. Then $\varphi$ is
called horizontally weakly conformal or semi conformal at $x$ if either

\begin{enumerate}
\item[(i)] $d\varphi_{x}=0$, or
\item[(ii)] $d\varphi_{x}$ maps horizontal space $\mathcal{H}%
_{x}=(ker(d\varphi_{x}))^\perp$ conformally onto $T_{\varphi_{*}}N$, i.e., $%
d\varphi_{x}$ is surjective and there exists a number $\Lambda(x)\neq0$ such
that
\begin{equation}
h(d\varphi_{x}X,d\varphi_{x}Y)=\Lambda(x)g(X,Y)\mbox{ }(X,Y\in\mathcal{H}%
_{x}).
\end{equation}
\end{enumerate}
\end{definition}
A point $x$ is of type (i) in Definition if and only if it is a critical point of $\varphi$; we shall call
a point of type (ii) a \textit{regular point}. At a critical point, $%
d\varphi_{x}$ has rank $0$; at a regular point, $d\varphi_{x}$ has rank $n$
and $\varphi$ is submersion. The number $\Lambda(x)$ is called the \textit{%
square dilation} (of $\varphi$ at $x$); it is necessarily non-negative; its
square root $\lambda(x)=\sqrt{\Lambda(x)}$ is called the dilation (of $%
\varphi$ at $x$). The map $\varphi$ is called \textit{horizontally weakly
conformal} or \textit{semi conformal} (on $M$) if it is horizontally weakly
conformal at every point of $M$. It is clear that if $\varphi$ has no
critical points, then we call it a (\textit{horizontally}) conformal
submersion.

Next, we recall the following definition from \cite{G}. Let
$\pi:M\longrightarrow N$ be a submersion. A vector field $E$ on $M$ is said
to be projectable if there exists a vector field $\check{E}$ on $N$, such
that $d\pi(E_{x})=\check{E}_{\pi(x)}$ for all $x\in M$. In this case $E$ and
$\check{E}$ are called $\pi-$ related. A horizontal vector field $Y$ on $%
(M,g)$ is called basic, if it is projectable. It is well known fact, that is
$\check{Z}$ is a vector field on $N$, then there exists a unique basic
vector field $Z$ on $M$, such that $Z$ and $\check{Z}$ are $\pi-$ related.
The vector field $Z$ is called the horizontal lift of $\check{Z}$.

The fundamental tensors of a submersion were introduced in \cite{O}. They
play a similar role to that of the second fundamental form of an immersion.
More precisely, O'Neill's tensors $T$ and $A$ defined for vector fields $E,F$
on $M$ by
\begin{equation}  \label{A}
A_{E}F=\mathcal{V}\nabla_{\mathcal{H}E}\mathcal{H}F +\mathcal{H}\nabla_{%
\mathcal{H}E}\mathcal{V}F
\end{equation}
\begin{equation}  \label{T}
T_{E}F=\mathcal{H}\nabla_{\mathcal{V}E}\mathcal{V}F +\mathcal{V}\nabla_{%
\mathcal{V}E}\mathcal{H}F
\end{equation}
where $\mathcal{V}$ and $\mathcal{H}$ are the vertical and horizontal
projections (see \cite{FIP}). On the other hand, from (\ref{A}) and (\ref{T}%
), we have
\begin{equation}  \label{nvw}
\nabla_{V}W=T_{V}W+\hat{\nabla}_{V}W
\end{equation}
\begin{equation}  \label{nvx}
\nabla_{V}X=\mathcal{H}\nabla_{V}X+T_{V}X
\end{equation}
\begin{equation}  \label{nxv}
\nabla_{X}V=A_{X}V +\mathcal{V}\nabla_{X}V
\end{equation}
\begin{equation}  \label{nxy}
\nabla_{X}Y=\mathcal{H}\nabla_{X}Y+A_{X}Y
\end{equation}
for $X,Y\in\Gamma((ker\pi_{*})^\perp)$ and $V,W\in\Gamma(ker\pi_{*})$, where $\hat{%
\nabla}_{V}W=\mathcal{V}\nabla_{V}W$. If $X$ is basic, then $\mathcal{H}%
\nabla_{V}X=A_{X}V$. It is easily seen that for $x\in M$, $X\in \mathcal{H}%
_{x}$ and $\mathcal{V}_{x}$ the linear operators $T_{V}$, $%
A_{X}:T_{X}M\longrightarrow T_{X}M$ are skew-symmetric, that is
\begin{equation*}
-g(T_{V}E,F)=g(E,T_{V}F)\text{ and }-g(A_{X}E,F)=g(E,A_{X}F)
\end{equation*}
for all $E,F\in T_{x}M$. We also see that the restriction of $T$ to the
vertical distribution $T\mid_{V\times V}$ is exactly the second fundamental
form of the fibres of $\pi$. Since $T_{V}$ skew-symmetric we get: $\pi$ has
totally geodesic fibres if and only if $T\equiv0$. For the special case when
$\pi$ is horizontally conformal we have the following:

\begin{proposition}
\textit{(\cite{G})} Let $\pi:(M^{m},g)\longrightarrow (N^{n},h)$ be a
horizontally conformal submersion with dilation $\nabla$ and $X,Y$ be
horizontal vectors, then
\begin{equation}
A_{X}Y=\frac{1}{2}\{\mathcal{V}[X,Y]-\lambda^{2}g(X,Y)grad_{\mathcal{V}}(%
\frac{1}{\lambda^{2}})\}.
\end{equation}
\end{proposition}

We see that the skew-symmetric part of $A\mid_{(ker\pi_{*})^\perp\times(ker\pi_{*})^\perp}$
measures the obstruction integrability of the horizontal distribution $%
(ker\pi_{*})^\perp$.

Let $(M,g_{M})$ and $(N,g_{N})$ be Riemannian manifolds and suppose that $%
\pi:M \longrightarrow N$ is a smooth map between them. The
differential $\pi_{*}$ of $\pi$ can be viewed a section of the
bundle $Hom(TM,\pi^{-1}TN) \longrightarrow M$, where $\pi^{-1}TN$ is
the pullback bundle which has fibres $(\pi^{-1}TN)_{p}=T_{\pi(p)}N$,
$p\in M$. $Hom(TM,\pi^{-1}TN)$ has a connection $\nabla$ induced from
the Levi-Civita connection $\nabla^{M}$ and the pullback connection. Then
the second fundamental form of $\pi$ is given by
$$\nabla\pi_*:\Gamma(TM)\times\Gamma(TM)\longrightarrow \Gamma(TN)$$ defined by
\begin{equation}
(\nabla\pi_{*})(X,Y)=\nabla^{\pi}_{X}\pi_{*}(Y)-\pi_{*}(\nabla^{M}_{X}Y) \label{nfixy}
\end{equation}
for $X,Y\in\Gamma(TM)$, where $\nabla^{\pi}$ is the pullback connection.
It is known that the second fundamental form is symmetric.
\begin{lemma}\cite{Urakawa}
Let $(M,g_{M})$ and $(N,g_{N})$ be Riemannian manifolds and suppose that $%
\varphi:M \longrightarrow N$ is a smooth map between them. Then we have
\begin{equation}
\nabla^{\varphi}_{X}\varphi_{*}(Y)-\nabla^{\varphi}_{Y}\varphi_{*}(X)-\varphi_{*}([X,Y])=0
\end{equation}
for $X,Y\in\Gamma(TM)$.
\end{lemma}
A smooth map $\pi:(M,g_{M}) \longrightarrow (N,g_{N})$ is said to be harmonic if $%
trace(\nabla\pi_{*})=0$. On the other hand, the tension field of $\pi
$ is the section $\tau(\pi)$ of $\Gamma(\pi^{-1}TN)$ defined by
\begin{equation}
\tau(\pi)=div\pi_{*}=\sum_{i=1}^{m}(\nabla\pi_{*})(e_{i},e_{i}),
\end{equation}
where $\{e_{1},...,e_{m}\}$ is the orthonormal frame on $M$. Then it follows
that $\pi$ is harmonic if and only if $\tau(\pi)=0$ (for details,
see \cite{BW}). Finally, we recall the following lemma from \cite{BW}.

\begin{lemma}
\label{lem1} Suppose that $\pi:M \longrightarrow N$ is a horizontally conformal submersion. Then, for any
horizontal vector fields $X,Y$ and vertical fields $V,W$ we have
\begin{enumerate}
\item[(i)] $(\nabla \pi_*)(X,Y)=X(\ln\lambda)\pi_* Y+Y(\ln\lambda)\pi_* X-g(X,Y)\pi_*(grad\ln\lambda)$;
\item[(ii)] $(\nabla \pi_*)(V,W)=-\pi_* (T_{V}W)$;
\item[(iii)] $(\nabla \pi_*)(X,V)=-\pi_*(\nabla^{M}_{X}V)=-\pi_*(A_{X}V)$.
\end{enumerate}
\end{lemma}

\section{Conformal Anti-invariant $\xi^\perp-$submersions}

In this section, we define conformal anti-invariant $\xi^\perp-$submersions from an almost
contact metric manifold onto a Riemannian manifold and investigate the integrability
of distributions and obtain a necessary and sufficient condition for such
submersions to be totally geodesic map. We also investigate the harmonicity
of such submersions.

\begin{definition}
\label{def} Let $(M,\phi,\xi,\eta,g_M)$ be an almost contact metric manifold and
and $(N,g_{N})$ be a Riemannian manifold. We suppose that there exist a horizontally conformal submersion
$\pi:M \longrightarrow N$ such that $\xi$ is normal to $ker\pi_{*}$ and $ker\pi_{*}$ is anti-invariant with respect to $\phi$, i.e.,
$\phi(ker\pi_{*})\subset(ker\pi_{*})^{\perp}.$ Then we say that $\pi$ is a conformal anti-invariant $\xi^\perp-$submersion.
\end{definition}

Here, we assume that if $\pi:(M,\phi,\xi,\eta,g_M)\longrightarrow (N,g_{N})$ is a conformal anti-invariant $\xi^\perp-$submersion from a Sasakian manifold
$(M,\phi,\xi,\eta,g_M)$ to a Riemannian manifold $(N,g_{N})$. Then from Definition \ref{def},
we have $\phi(ker\pi_{*})^\perp\cap ker\pi_{*}\neq{0}.$ We denote the complementary orthogonal
distribution to $\phi(ker\pi_{*})$ in $(ker\pi_{*})^\perp$ by $\mu.$ Then we have
\begin{equation}
(ker\pi_{*})^\perp=\phi(ker\pi_{*})\oplus\mu. \label{e.q:3.1}
\end{equation}
We can easily to see that $\mu$ is an invariant distribution of $(ker\pi_{*})^\perp,$ with respect to $\phi.$ Hence $\mu$ contains $\xi.$ Thus, for
$X\in\Gamma((ker\pi_{*})^\perp)$, we have
\begin{equation}
\phi X=\mathcal{B}X+\mathcal{C}X, \label{e.q:3.2}
\end{equation}
where $\mathcal{B}X\in\Gamma(ker\pi_{*})$ and $\mathcal{C}X\in\Gamma(\mu).$ On the other hand, since $\pi_{*}((ker\pi_{*})^\perp)=TN$ and $\pi$ is a conformal submersion, using (\ref{e.q:3.2}) we derive $\frac{1}{\lambda^2}g_{N}(\pi_{*}\phi V,\pi_{*}\mathcal{C}X)=0$ for any $X\in\Gamma((ker\pi_{*})^\perp)$ and
$V\in\Gamma(ker\pi_{*}),$ which implies that
\begin{equation}
TN=\pi_{*}(\phi ker\pi_{*})\oplus \pi_{*}(\mu). \label{e.q:3.3}
\end{equation}
\begin{remark}
We note that every anti-invariant $\xi^\perp-$submersion from an almost contact manifold onto
a Riemannian manifold is a conformal anti-invariant $\xi^\perp-$submersion with $\lambda=I$, where $I$ denotes the identity function \cite{Lee}.
\end{remark}
\begin{lemma}
Let $\pi$ be a conformal anti-invariant $\xi^\perp$-submersion from a Sasakian manifold $(M,\phi,\xi,\eta,g_{M})$ onto a Riemannian manifold $(N,g_{N})$. Then we have
\begin{equation}
A_X\xi=-\mathcal{B}X,  \label{axxi}
\end{equation}
\begin{equation}
T_V\xi=0  \label{tvxi},
\end{equation}
\begin{equation}
g_{M}(\mathcal{C}Y,\phi V)=0, \label{cywv}
\end{equation}
\begin{equation}
g_{M}(\nabla_{X}\mathcal{C}Y,\phi V)=-g_{M}(\mathcal{C}Y,\phi A_{X}V)  \label{nxcyjv}
\end{equation}
for $X,Y,\xi\in\Gamma((ker\pi_{*})^\perp)$ and $V\in\Gamma(ker\pi_{*})$.
\end{lemma}
\begin{proof}
By virtue of (\ref{e.q:2.3}), (\ref{nxy}) and (\ref{e.q:3.2}) we have (\ref{axxi}). Using (\ref{e.q:2.3}) and (\ref{nvx}) we get (\ref{tvxi}).
By using (\ref{e.q:2.2}), for $Y\in\Gamma((ker\pi_{*})^\perp)$ and $V\in\Gamma(ker\pi_{*})$, we have
$$g_{M}(\mathcal{C}Y,\phi V)=g_{M}(\phi Y-\mathcal{B}Y,\phi V)=g_{M}(\phi Y,\phi V)=g_{M}(Y,V)+\eta(Y)\eta(V)=g_{M}(Y,V)=0,$$ since
$\mathcal{B}Y\in\Gamma(ker\pi_{*})$ and $\phi V,\xi\in\Gamma((ker\pi_{*})^\perp).$ Differentiating (\ref{cywv}) with respect to $X,$ we get
\begin{align*}
g_{M}(\nabla_{X}\mathcal{C}Y,\phi V)&=-g_{M}(\mathcal{C}Y,\nabla_{X}\phi V)\\
&=-g_{M}(\mathcal{C}Y,(\nabla_{X}\phi) V)-g_{M}(\mathcal{C}Y,\phi(\nabla_{X}V))\\
&=-g_{M}(\mathcal{C}Y,\phi(\nabla_{X}V))\\
&=-g_{M}(\mathcal{C}Y,\phi A_{X}V)-g_{M}(\mathcal{C}Y,\phi\mathcal{V}\nabla_{X}V)\\
&=-g_{M}(\mathcal{C}Y,\phi A_{X}V)
\end{align*}
due to $\phi\mathcal{V}\nabla_{X}V\in\Gamma(\phi ker\pi_{*}).$ Our assertion is complete.
\end{proof}
Since the distribution $ker\pi_{*}$ is integrable, we only study the integrability of the distribution $(ker\pi_{*})^\perp$ and
then we investigate the geometry of leaves of $ker\pi_{*}$ and $(ker\pi_{*})^\perp$.

\begin{theorem}\label{teo1}
Let $\pi:(M,\phi,\xi,\eta,g_M)\longrightarrow (N,g_{N})$ be a conformal anti-invariant $\xi^\perp-$submersion from a Sasakian manifold $(M,\phi,\xi,\eta,g_{M})$ onto a Riemannian manifold $(N,g_{N})$. Then the following assertions are equivalent to each other;
\begin{enumerate}
\item[(a)] $(ker\pi_{*})^\perp$ is integrable,
\item[(b)]$\begin{aligned}[t]
\frac{1}{\lambda^{2}}g_{N}(\nabla^{\pi}_{Y}\pi_{*}\mathcal{C}X-\nabla^{\pi}_{X}\pi_{*}\mathcal{C}Y,\pi_{*}\phi V)
&=g_{M}(A_{X}\mathcal{B}Y-A_{Y}\mathcal{B}X-\mathcal{C}Y(\ln\lambda)X+\mathcal{C}X(\ln\lambda)Y\\
&-2g_M(\mathcal{C}X,Y)\ln\lambda-\eta(Y)X+\eta(X)Y,\phi V)
\end{aligned}$
\end{enumerate}
for $X,Y\in\Gamma((ker\pi_{*})^\perp)$ and $V\in\Gamma(ker\pi_{*})$.
\end{theorem}
\begin{proof}
From  (\ref{e.q:2.2}) and (\ref{e.q:2.3}), we obtain
\begin{equation}
g_{M}(\nabla_{X}Y,V)=g_{M}(\nabla_{X}\phi Y,\phi V)-\eta(Y)g_{M}(X,\phi V). \label{e.q:3.8}
\end{equation}
for $X,Y\in\Gamma((ker\pi_{*})^\perp)$ and $V\in\Gamma(ker\pi_{*})$. Then, from (\ref{e.q:3.2}) and (\ref{e.q:3.8}), we have
\begin{align*}
g_{M}([X,Y],V)&=g_{M}(\nabla_{X}\phi Y,\phi V)-g_{M}(\nabla_{Y}\phi X,\phi V)
-\eta(Y)g_{M}(X,\phi V)+\eta(X)g_{M}(Y,\phi V)\\
&=g_{M}(\nabla_{X}\mathcal{B}Y,\phi V)+g_{M}(\nabla_{X}\mathcal{C}Y,\phi V)
-g_{M}(\nabla_{Y}\mathcal{B}X,\phi V)-g_{M}(\nabla_{Y}\mathcal{C}X,\phi V)\\
&-\eta(Y)g_{M}(X,\phi V)+\eta(X)g_{M}(Y,\phi V).
\end{align*}
Using (\ref{nxv}) and if we take into account that $\pi$ is a conformal submersion, we obtain
\begin{align*}
g_{M}([X,Y],V)&=g_{M}(A_{X}\mathcal{B}Y-A_{Y}\mathcal{B}X,\phi V)+\frac{1}{\lambda^{2}}g_{N}(\pi_{*}(\nabla_{X}\mathcal{C}Y),\pi_{*}\phi V)\\
&-\frac{1}{\lambda^{2}}g_{N}(\pi_{*}(\nabla_{Y}\mathcal{C}X),\pi_{*}\phi V)-\eta(Y)g_{M}(X,\phi V)+\eta(X)g_{M}(Y,\phi V).
\end{align*}
Thus, from (\ref{nfixy}) and Lemma \ref{lem1} we derive
\begin{align*}
g_{M}([X,Y],V)&=g_{M}(A_{X}\mathcal{B}Y-A_{Y}\mathcal{B}X,\phi V)-g_{M}(\mathcal{H}gradln\lambda,X)g_{M}(\mathcal{C}Y,\phi V)\\
&-g_{M}(\mathcal{H}gradln\lambda,\mathcal{C}Y)g_{M}(X,\phi V)+g_{M}(X,\mathcal{C}Y)g_{M}(\mathcal{H}gradln\lambda,\phi V)\\
&+\frac{1}{\lambda^{2}}g_{N}(\nabla^{\pi}_X\pi_{*}\mathcal{C}Y,\pi_{*}\phi V)+g_{M}(\mathcal{H}gradln\lambda,Y)g_{M}(\mathcal{C}X,\phi V)\\
&+g_{M}(\mathcal{H}gradln\lambda,\mathcal{C}X)g_{M}(Y,\phi V)-g_{M}(Y,\mathcal{C}X)g_{M}(\mathcal{H}gradln\lambda,\phi V)\\
&-\frac{1}{\lambda^{2}}g_{N}(\nabla^{\pi}_Y\pi_{*}\mathcal{C}X,\pi_{*}\phi V)-\eta(Y)g_{M}(X,\phi V)+\eta(X)g_{M}(Y,\phi V).
\end{align*}
Moreover, using (\ref{cywv}), we obtain
\begin{align*}
g_{M}([X,Y],V)&=g_{M}(A_{X}\mathcal{B}Y-A_{Y}\mathcal{B}X-\mathcal{C}Y(\ln\lambda)X+\mathcal{C}X(\ln\lambda)Y
-2g_M(\mathcal{C}X,Y)\ln\lambda\\
&-\eta(Y)X+\eta(X)Y,\phi V)-\frac{1}{\lambda^{2}}g_{N}(\nabla^{\pi}_Y\pi_{*}\mathcal{C}X-\nabla^{\pi}_X\pi_{*}\mathcal{C}Y,\pi_{*}\phi V).
\end{align*}
This show that $(a)\Leftrightarrow (b)$.
\end{proof}
From Theorem $\ref{teo1},$ we deduce the following which shows that a conformal anti-invariant $\xi^\perp-$submersion
with integrable $(ker\pi_{*})^{\perp}$ turns out to be a horizontally homothetic submersion.
\begin{theorem}
Let $\pi$ be a conformal anti-invariant $\xi^\perp-$submersion from a Sasakian manifold $(M,\phi,\xi,\eta,g_{M})$ onto a Riemannian manifold $(N,g_{N})$. Then any two conditions below imply the third;
\begin{enumerate}
\item [(i)] $(ker\pi_{*})^\perp$ is integrable.
\item [(ii)] $\pi$ is horizontally homothetic submersion.
\item [(iii)]$\begin{aligned}[t]
g_{N}(\nabla^{\pi}_Y\pi_{*}\mathcal{C}X-\nabla^{\pi}_X\pi_{*}\mathcal{C}Y,\pi_{*}\phi V)=\lambda^{2}g_{M}(A_{X}\mathcal{B}Y-A_{Y}\mathcal{B}X-\eta(Y)X+\eta(X)Y,\phi V)
\end{aligned}$
\end{enumerate}
for $X,Y\in\Gamma((ker\pi_{*})^\perp)$ and $V\in\Gamma(ker\pi_{*})$.
\end{theorem}
\begin{proof}
From Theorem \ref{teo1}, we have
\begin{align*}
g_{M}([X,Y],V)&=g_{M}(A_{X}\mathcal{B}Y-A_{Y}\mathcal{B}X-\mathcal{C}Y(\ln\lambda)X+\mathcal{C}X(\ln\lambda)Y
-2g_M(\mathcal{C}X,Y)\ln\lambda\\
&-\eta(Y)X+\eta(X)Y,\phi V)-\frac{1}{\lambda^{2}}g_{N}(\nabla^{\pi}_Y\pi_{*}\mathcal{C}X-\nabla^{\pi}_X\pi_{*}\mathcal{C}Y,\pi_{*}\phi V)
\end{align*}
for $X,Y\in\Gamma((ker\pi_{*})^\perp)$ and $V\in\Gamma(ker\pi_{*})$. Now, if we have (i) and (iii), then we arrive at
\begin{align}
&-g_{M}(\mathcal{H}grad\ln\lambda,\mathcal{C}Y)g_{M}(X,\phi V)+g_{M}(\mathcal{H}grad\ln\lambda,\mathcal{C}X)g_{M}(Y,\phi V)\label{e.q:3.9} \\
&-2g_{M}(\mathcal{C}X,Y)g_{M}(\mathcal{H}grad\ln\lambda,\phi V)=0.\notag
\end{align}
Now, taking $Y=\phi V$ in (\ref{e.q:3.9}) for $V\in\Gamma(ker\pi_{*})$, using (\ref{e.q:2.2}) and (\ref{cywv}), we get
\begin{align*}
g_{M}(\mathcal{H}grad\ln\lambda,\mathcal{C}X)g_{M}(\phi V, \phi V))&=g_{M}(\mathcal{H}grad\ln\lambda,\mathcal{C}X)\{g_{M}(V,V)-\eta(V)\eta(V)\}\\
&=g_{M}(\mathcal{H}grad\ln\lambda,\mathcal{C}X)g_{M}(V,V)=0.
\end{align*}
Hence $\lambda$ is a constant on $\Gamma(\mu)$. On the other hand, taking $Y=\mathcal{C}X$ in (\ref{e.q:3.9})
for $X\in\Gamma(\mu)$ and using (\ref{cywv}) we derive
\begin{align*}
&-g_{M}(\mathcal{H}grad\ln\lambda,\mathcal{C}^{2}Y)g_{M}(X,\phi V)+g_{M}(\mathcal{H}grad\ln\lambda,\mathcal{C}X)g_{M}(\mathcal{C}X,\phi V)\\
&-2g_{M}(\mathcal{C}X,\mathcal{C}X)g_{M}(\mathcal{H}grad\ln\lambda,\phi V)=0,
\end{align*}
thus, we arrive at $$2g_{M}(\mathcal{C}X,\mathcal{C}X)g_{M}(\mathcal{H}grad\ln\lambda,\phi V)=0.$$ From above equation, $\lambda$ is a constant on $\Gamma(\phi ker\pi_{*}).$ Similarly, one can obtain the other assertions.
\end{proof}
\begin{remark}
We assume that $(ker\pi_*)^\perp=\phi ker\pi_*\oplus\{\xi\}.$ Using (\ref{e.q:3.2}) one can prove that $\mathcal{C}X=0.$
\end{remark}
Hence we have the following corollary.
\begin{corollary}
Let $\pi$ be a conformal anti-invariant $\xi^\perp-$submersion from a Sasakian manifold $(M,\phi,\xi,\eta,g_{M})$ onto a Riemannian manifold $(N,g_{N})$ with
$(ker\pi_{*})^{\perp}=\phi(ker\pi_{*})\oplus<\xi>$. Then the following assertions are equivalent to each other;
\begin{enumerate}
\item [(i)] $(ker\pi_{*})^\perp$ is integrable
\item [(ii)] $A_{X}\phi Y+\eta(X)Y=A_{Y}\phi X+\eta(Y)X$
\item [(iii)] $(\nabla \pi_{*})(X,\phi Y)+\eta(Y)\pi_{*}X=(\nabla \pi_{*})(Y,\phi X)+\eta(X)\pi_{*}Y$
\end{enumerate}
for $X,Y\in\Gamma((ker\pi_{*})^\perp)$.
\end{corollary}
For the geometry of leaves of the horizontal distribution, we have the following theorem.
\begin{theorem}\label{teo2}
Let $\pi:(M,\phi,\xi,\eta,g_M)\longrightarrow (N,g_{N})$ be a conformal anti-invariant $\xi^\perp$-submersion from a Sasakian manifold $(M,\phi,\xi,\eta,g_{M})$ onto a Riemannian manifold $(N,g_{N})$. Then the following assertions are equivalent to each other;
\begin{enumerate}
\item [(i)] $(ker\pi_{*})^\perp$ defines a totally geodesic foliation on $M$.
\item [(ii)]$\begin{aligned}[t]
-\frac{1}{\lambda^{2}}g_{N}(\nabla^{\pi}_{X}\pi_{*}\mathcal{C}Y,\pi_{*}\phi V)=g_{M}(A_{X}BY-\mathcal{C}Y(\ln\lambda)X+g_M(X,\mathcal{C}Y)\ln\lambda-\eta(Y)X,\phi V)
\end{aligned}$
\end{enumerate}
for $X,Y\in\Gamma((ker\pi_{*})^\perp)$ and $V\in\Gamma(ker\pi_{*})$.
\end{theorem}
\begin{proof}
 By using (\ref{e.q:2.2}), (\ref{nxv}), (\ref{nxy}), (\ref{e.q:3.1}), (\ref{e.q:3.2}) and (\ref{e.q:3.8}), have
\begin{align*}
g_{M}(\nabla_{X}Y,V)&=g_{M}(A_{X}BY,\phi V)+g_{M}(\nabla_{X}\mathcal{C}Y,\phi V)-\eta(Y)g_{M}(X,\phi V)
\end{align*}
for $X,Y\in\Gamma((ker\pi_{*})^\perp)$ and $V\in\Gamma(ker\pi_{*})$. Since $\pi$ is a conformal submersion, using (\ref{nfixy}) and Lemma (\ref{lem1}) we arrive at
\begin{align*}
g_{M}(\nabla_{X}Y,V)&=g_{M}(A_{X}BY,\phi V)-\frac{1}{\lambda^2}g_{M}(\mathcal{H}grad\ln\lambda,X)g_{N}(\pi_{*}\mathcal{C}Y,\pi_{*}\phi V)\\
&-\frac{1}{\lambda^2}g_{M}(\mathcal{H}grad\ln\lambda,\mathcal{C}Y)g_{N}(\pi_{*}X,\pi_{*}\phi V)\\
&+\frac{1}{\lambda^2}g_{M}(X,\mathcal{C}Y)g_{N}(\pi_{*}(\mathcal{H}grad\ln\lambda),\pi_{*}\phi V)\\
&+\frac{1}{\lambda^2}g_{N}(\nabla^{\pi}_X\pi_{*}\mathcal{C}Y,\pi_{*}\phi V)-\eta(Y)g_{M}(X,\phi V).
\end{align*}
Moreover, using Definiton \ref{def} and (\ref{cywv}) we obtain
\begin{align*}
g_{M}(\nabla_{X}Y,V)&=g_{M}(A_{X}BY-\mathcal{C}Y(\ln\lambda)X+g_M(X,\mathcal{C}Y)\ln\lambda-\eta(Y)X,\phi V)\\
&+\frac{1}{\lambda^2}g_{N}(\nabla_{\pi_{*}X}\pi_{*}\mathcal{C}Y,\pi_{*}\phi V)
\end{align*}
which tells that $(i)\Leftrightarrow(ii)$.
\end{proof}
From Theorem \ref{teo2}, we also deduce the following characterization.
\begin{theorem}
Let $\pi$ be a conformal anti-invariant $\xi^\perp-$submersion from a Sasakian manifold $(M,\phi,\xi,\eta,g_{M})$
 onto a Riemannian manifold $(N,g_{N})$. Then any two conditions below imply the third;
\begin{enumerate}
\item [(i)] $(ker\pi_{*})^\perp$ defines a totally geodesic foliation on $M$.
\item [(ii)] $\pi$ is a horizontally homothetic submersion.
\item [(iii)] $g_{N}(\nabla^{\pi}_X\pi_{*}\mathcal{C}Y,\pi_{*}\phi V)=\lambda^{2}g_{M}(-A_{X}\mathcal{B}Y+\eta(Y)X,\phi V)$
\end{enumerate}
for $X,Y\in\Gamma((ker\pi_{*})^\perp)$ and $V\in\Gamma(ker\pi_{*})$.
\end{theorem}
\begin{proof}
For $X,Y\in\Gamma((ker\pi_{*})^\perp)$ and $V\in\Gamma(ker\pi_{*})$, from Theorem \ref{teo2}, we have
\begin{align*}
g_{M}(\nabla_{X}Y,V)&=g_{M}(A_{X}BY-\mathcal{C}Y(\ln\lambda)X+g_M(X,\mathcal{C}Y)\ln\lambda-\eta(Y)X,\phi V)\\
&+\frac{1}{\lambda^2}g_{N}(\nabla^{\pi}_X\pi_{*}\mathcal{C}Y,\pi_{*}\phi V).
\end{align*}
Now, if we have (i) and (iii), then we obtain
\begin{equation}
-g_{M}(\mathcal{H}grad\ln\lambda,\mathcal{C}Y)g_{M}(X,\phi V)+g_{M}(\mathcal{H}grad\ln\lambda,\phi V)g_{M}(X,\mathcal{C}Y)=0. \label{e.q:3.10}
\end{equation}
Now, taking $X=\mathcal{C}Y)$ in (\ref{e.q:3.10}) and using (\ref{cywv}), we get
$g_{M}(\mathcal{H}grad\ln\lambda,\phi V)g_{M}(X,\mathcal{C}Y)=0.$ Hence, $\lambda$ is a constant on $\Gamma(\phi ker\pi_{*})$. On the other hand, taking $X=\phi V$ in (\ref{e.q:3.10}) and using (\ref{cywv}) we derive
\begin{align*}
g_{M}(\mathcal{H}grad\ln\lambda,\mathcal{C}Y)g_{M}(\phi V, \phi V))&=g_{M}(\mathcal{H}grad\ln\lambda,\mathcal{C}Y)\{g_{M}(V,V)-\eta(V)\eta(V)\}\\
&=g_{M}(\mathcal{H}grad\ln\lambda,\mathcal{C}Y)g_{M}(V,V)=0.
\end{align*}
From above equation, $\lambda$ is a constant on $\Gamma(\mu)$. Similarly, one can obtain the other assertions.
\end{proof}
In particular, as an analogue of a conformal Lagrangian submersion in \cite{As}, we have the following corollary.
\begin{corollary}\label{cor1}
Let $\pi$ be a conformal anti-invariant $\xi^\perp-$submersion from a Sasakian manifold $(M,\phi,\xi,\eta,g_{M})$ onto a Riemannian manifold $(N,g_{N})$ with
$(ker\pi_{*})^{\perp}=\phi(ker\pi_{*})\oplus<\xi>$. Then the following assertions are equivalent to each other;
\begin{enumerate}
\item [(i)] $(ker\pi_{*})^\perp$ defines a totally geodesic foliation on $M$.
\item [(ii)] $A_{X}\mathcal{B}Y=\eta(Y)X$
\item [(iii)]$(\nabla \pi_{*})(X,\phi V)=-\eta(Y)\pi_{*}X$
\end{enumerate}
for $X,Y\in\Gamma((ker\pi_{*})^\perp)$ and $V\in\Gamma(ker\pi_{*})$.
\end{corollary}

In the sequel we are going to investigate the geometry of leaves of the distribution $ker\pi_{*}$.
\begin{theorem}\label{teo3}
Let $\pi$ be a conformal anti-invariant $\xi^\perp-$submersion from a Sasakian manifold $(M,\phi,\xi,\eta,g_{M})$ onto a Riemannian manifold $(N,g_{N})$. Then the following assertions are equivalent to each other;
\begin{enumerate}
\item [(i)] $ker\pi_{*}$ defines a totally geodesic foliation on $M$.
\item [(ii)]$\begin{aligned}[t]
-\frac{1}{\lambda^2}g_{N}(\nabla^{\pi}_{\phi W}\pi_{*}\phi V,\pi_{*}\phi\mathcal{C}X)&=g_{M}(\phi\mathcal{C}X(\ln\lambda)\phi V-T_{V}\mathcal{B}X, \phi V)+\eta(\nabla_{\phi W}V)\eta(\mathcal{C}X)
\end{aligned}$
\end{enumerate}
for $V,W\in\Gamma(ker\pi_{*})$ and $X\in\Gamma((ker\pi_{*})^\perp)$.
\end{theorem}
\begin{proof}
Since $g_{M}(W,\xi)=0,$ using (\ref{e.q:2.3}) we have $g_{M}(\nabla_{V}W,\xi)=-g_{M}(W,\nabla_{V}\xi)=-g_{M}(W,\phi V)=0$
for $V,W\in\Gamma(ker\pi_{*})$ and $\xi\in\Gamma((ker\pi_{*})^\perp).$  Thus we have
\begin{align*}
g_{M}(\nabla_{V}W,X)&=g_{M}(\phi\nabla_{V}W,\phi X)+\eta(\nabla_{V}W)\eta(X)\\
&=g_{M}(\phi\nabla_{V}\phi W,\phi X)\\
&=g_{M}(\nabla_{V}\phi W,\phi X)-g_{M}((\nabla_{V}\phi)W,\phi X).
\end{align*}
Using (\ref{e.q:2.3}), (\ref{nvw}) and (\ref{e.q:3.2}) we have
\begin{equation*}
g_{M}(\nabla_{V}W,X)=g_{M}(T_{V}\phi W,\mathcal{B}X)+g_{M}(\mathcal{H}\nabla_{V}\phi W,\mathcal{C}X).
\end{equation*}
Since $\nabla$ is torsion free and $[V,\phi W]\in\Gamma(ker\pi_{*})$ we obtain
\begin{equation*}
g_{M}(\nabla_{V}W,X)=g_{M}(T_{V}\phi W,\mathcal{B}X)+g_{M}(\nabla_{\phi W}V,\mathcal{C}X).
\end{equation*}
Using (\ref{e.q:2.3}) and (\ref{nxy}) we have
\begin{align*}
g_{M}(\nabla_{V}W,X)&=g_{M}(T_{V}\phi W,\mathcal{B}X)+g_{M}(\phi\nabla_{\phi W}V,\phi\mathcal{C}X)+\eta(\nabla_{\phi W}V)\eta(\mathcal{C}X)\\
&=g_{M}(T_{V}\phi W,\mathcal{B}X)+g_{M}(\nabla_{\phi W}\phi V,\phi\mathcal{C}X)+\eta(\nabla_{\phi W}V)\eta(\mathcal{C}X)
\end{align*}
here we have used that $\mu$ is invariant. Using (\ref{nfixy}) and Lemma \ref{lem1} (i) and if we take into account
that $\pi$ is a conformal submersion, we obtain
\begin{align*}
g_{M}(\nabla_{U}V,X)&=g_{M}(T_{V}\phi W,\mathcal{B}X)+\frac{1}{\lambda^2}g_{M}(\mathcal{H}gradln\lambda,\phi W)g_{N}(\pi_{*}\phi V,\pi_{*}\phi\mathcal{C}X)\\
&-\frac{1}{\lambda^2}g_{M}(\mathcal{H}gradln\lambda,\phi V)g_{N}(\pi_{*}\phi W,\pi_{*}\phi\mathcal{C}X)\\
&+g_{M}(\phi W,\phi V)\frac{1}{\lambda^2}g_{N}(\pi_{*}(\mathcal{H}gradln\lambda),\pi_{*}\phi\mathcal{C}X)\\
&+\frac{1}{\lambda^2}g_{N}(\nabla_{\pi_{*}\phi W}\pi_{*}\phi V,\pi_{*}\phi\mathcal{C}X)+\eta(\nabla_{\phi W}V)\eta(\mathcal{C}X).
\end{align*}
Moreover, using Definition \ref{def} and (\ref{cywv}), we obtain
\begin{align*}
g_{M}(\nabla_{U}V,X)&=g_{M}(\phi\mathcal{C}X(\ln\lambda)\phi V-T_{V}\mathcal{B}X, \phi V)+\eta(\nabla_{\phi W}V)\eta(\mathcal{C}X)\\
&+\frac{1}{\lambda^2}g_{N}(\nabla^{\pi}_{\phi W}\pi_{*}\phi V,\pi_{*}\phi\mathcal{C}X)
\end{align*}
which tells that $(i)\Leftrightarrow (ii)$.
\end{proof}
From Theorem \ref{teo3}, we deduce the following result.
\begin{theorem}
Let $\pi$ be a conformal anti-invariant $\xi^\perp-$submersion from a Sasakian manifold $(M,\phi,\xi,\eta,g_{M})$ onto a Riemannian manifold $(N,g_{N})$. Then any two conditions below imply the third;
\begin{enumerate}
\item [(i)] $ker\pi_{*}$ defines a totally geodesic foliation on $M$.
\item [(ii)] $\lambda$ is a constant on $\Gamma(\mu)$.
\item [(iii)]$\begin{aligned}[t]
-\frac{1}{\lambda^2}g_{N}(\nabla^{\pi}_{\phi W}\pi_{*}\phi V,\pi_{*}\phi\mathcal{C}X)&=g_{M}(T_{V}\phi W,\mathcal{B}X)+\eta(\nabla_{\phi W}V)\eta(\mathcal{C}X)
\end{aligned}$
\end{enumerate}
for $V,W\in\Gamma(ker\pi_{*})$ and $X\in\Gamma((ker\pi_{*})^\perp)$.
\end{theorem}
\begin{proof}
From Theorem (\ref{teo3}) we have
\begin{align*}
g_{M}(\nabla_{U}V,X)&=g_{M}(\phi\mathcal{C}X(\ln\lambda)\phi V-T_{V}\mathcal{B}X, \phi V)+\eta(\nabla_{\phi W}V)\eta(\mathcal{C}X)\\
&+\frac{1}{\lambda^2}g_{N}(\nabla^{\pi}_{\phi W}\pi_{*}\phi V,\pi_{*}\phi\mathcal{C}X)
\end{align*}
for $U,V\in\Gamma(ker\pi_{*})$ and $X\in\Gamma((ker\pi_{*})^\perp).$ Now, if we have (i) and (iii), then we obtain
$$g_{M}(\phi W,\phi V)g_{M}(\mathcal{H}gradln\lambda,\phi\mathcal{C}X)=0.$$ From above equation, $\lambda$ is a constant on $\Gamma(\mu)$. Similarly, one can obtain the other assertions.
\end{proof}
As an analogue of a conformal Lagrangian submersion in \cite{As}, (\ref{e.q:3.3}) implies that $TN=\pi_{*}(\phi ker\pi_{*})$. Hence we have the following.
\begin{corollary}\label{cor2}
Let $\pi$ be a conformal anti-invariant $\xi^\perp-$submersion from a Sasakian manifold $(M,\phi,\xi,\eta,g_{M})$ onto a Riemannian manifold $(N,g_{N})$ with
$(ker\pi_{*})^{\perp}=\phi(ker\pi_{*})\oplus<\xi>$. Then the following assertions are equivalent to each other;
\begin{enumerate}
\item [(i)] $ker\pi_{*}$ defines a totally geodesic foliation on $M$.
\item [(ii)] $T_{V}\phi W=0$
\end{enumerate}
for $V,W\in\Gamma(ker\pi_{*})$ and $X\in\Gamma((ker\pi_{*})^\perp)$.
\end{corollary}

Now we obtain necessary and sufficient condition for conformal anti-invariant $\xi^\perp-$submersion to be totally geodesic. We note that a differentiable map $\pi$ between two Riemannian manifolds is called totally geodesic if $\nabla \pi_{*}=0.$ A geometric interpretation of a totally geodesic map is that it maps every geodesic in the total manifold into a geodesic in the base manifold in proportion to arc lengths.
\begin{theorem}
Let $\pi:(M,\phi,\xi,\eta,g_{M})\longrightarrow (N,g_{N})$ be a conformal anti-invariant $\xi^\perp-$submersion, where $(M,\phi,\xi,\eta,g_{M})$ is a Sasakian manifold and $(N,g_{N})$ is a Riemannian manifold. Then $\pi$ is a totally geodesic map if
\begin{align}
-\nabla^{\pi}_{X}\pi_{*}Y_2&=\pi_{*}(\phi(A_{X}\phi Y_{1}+\mathcal{V}\nabla_{X}\mathcal{B}Y_{2} +A_{X}\mathcal{C}Y_{2})+\mathcal{C}(\mathcal{H}\nabla_{X}\phi Y_{1}+A_{X}\mathcal{B}Y_{2}+\mathcal{H}\nabla_{X}\mathcal{C}Y_{2}))\label{e.q:3.11}\\
&-\eta(Y_{2})\pi_{*}X-\{g_{M}(X,\phi Y_1)+g_M(X,\mathcal{C}Y_2)\}\pi_{*}\xi\notag
\end{align}
for any $X\in\Gamma((ker\pi_*)^\perp), Y=Y_{1}+Y_{2}\in\Gamma(TM)$, where $Y_{1}\in\Gamma(ker\pi_{*})$ and $Y_{2}\in\Gamma((ker\pi_{*})^\perp)$.
\end{theorem}
\begin{proof}
By virtue of (\ref{e.q:2.2}) and (\ref{nfixy}) we have
\begin{align*}
(\nabla \pi_{*})(X,Y)&=\nabla^{\pi}_{X}\pi_{*}Y+\pi_{*}(-\nabla_{X}Y)\\
&=\nabla^{\pi}_{X}\pi_{*}Y+\pi_{*}(\phi\nabla_{X}\phi Y-g(X,\phi Y)\xi-\eta(Y)X)
\end{align*}
for any $X\in\Gamma((ker\pi_*)^\perp), Y\in\Gamma(TM)$. Then from (\ref{nxv}), (\ref{nxy}) and (\ref{e.q:3.2}) we get
\begin{align*}
(\nabla \pi_{*})(X,Y)&=\nabla^\pi_{X}\pi_{*}Y_2+\pi_{*}(\phi A_{X}\phi Y_{1}+\mathcal{B}\mathcal{H}\nabla_{X}\phi Y_{1}
+\mathcal{C}\mathcal{H}\nabla_{X}\phi Y_{1}+\mathcal{B}A_{X}\mathcal{B}Y_{2}\\
&+\mathcal{C}A_{X}\mathcal{B}Y_{2}+\phi\mathcal{V}\nabla_{X}\mathcal{B}Y_{2}
+\phi A_{X}\mathcal{C}Y_{2}+\mathcal{B}\mathcal{H}\nabla_{X}\mathcal{C}Y_{2}+\mathcal{C}\mathcal{H}\nabla_{X}\mathcal{C}Y_{2})\\
&-\eta(Y_{2})\pi_{*}X-\{g_{M}(X,\phi Y_1)+g_M(X,\mathcal{C}Y_2)\}\pi_{*}\xi
\end{align*}
for any $Y=Y_{1}+Y_{2}\in\Gamma(TM)$, where $Y_{1}\in\Gamma(ker\pi_{*})$ and $Y_{2}\in\Gamma((ker\pi_{*})^\perp)$. Thus taking into account the vertical parts, we find
\begin{align*}
(\nabla \pi_{*})(X,Y)&=\nabla^\pi_{X}\pi_{*}Y+\pi_{*}(\phi(A_{X}\phi Y_{1}+\mathcal{V}\nabla_{X}\mathcal{B}Y_{2}+A_{X}\mathcal{C}Y_{2})\\
&+\mathcal{C}(\mathcal{H}\nabla_{X}\phi Y_{1}+A_{X}\mathcal{B}Y_{2}+\mathcal{H}\nabla_{X}\mathcal{C}Y_{2}))\\
&-\eta(Y_{2})\pi_{*}X-\{g_{M}(X,\phi Y_1)+g_M(X,\mathcal{C}Y_2)\}\pi_{*}\xi.
\end{align*}
Thus $(\nabla \pi_{*})(X,Y)=0$ if and only if the equation (\ref{e.q:3.11}) is satisfied.\\
\end{proof}

We now present the following definition.
\begin{definition}
Let $\pi$ be a conformal anti-invariant $\xi^\perp-$submersion from a Sasakian manifold $(M,\phi,\xi,\eta,g_{M})$ onto a Riemannian manifold $(N,g_{N})$ with
$(ker\pi_{*})^{\perp}=\phi(ker\pi_{*})\oplus<\xi>$. Then $\pi$ is called a $(\phi ker\pi_{*},\mu)$-totally geodesic map if
$$(\nabla \pi_{*})(\phi U,\xi)=0, for \ U\in\Gamma(ker\pi_{*})\ and \ \xi\in\Gamma((ker\pi_{*})^\perp).$$
\end{definition}
In the sequel we show that this notion has an important effect on the character of the conformal submersion.
\begin{theorem}
Let $\pi$ be a conformal anti-invariant $\xi^\perp-$submersion from a Sasakian manifold $(M,\phi,\xi,\eta,g_{M})$ onto a Riemannian manifold $(N,g_{N})$ with
$(ker\pi_{*})^{\perp}=\phi(ker\pi_{*})\oplus<\xi>$. Then $\pi$ is a $(\phi ker\pi_{*},\mu)$-totally geodesic map if and only if $\pi$ is a horizontally homothetic map.
\end{theorem}
\begin{proof}
For $U\in\Gamma(ker\pi_{*})$ and $\xi\in\Gamma(\mu)$, from Lemma \ref{lem1}, we have
\begin{equation*}
(\nabla \pi_{*})(\phi U,\xi)=\phi U(\ln\lambda)\pi_{*}\xi+\xi(\ln\lambda)\pi_{*}\phi U-g_{M}(\phi U,\xi)\pi_{*}(grad\ln\lambda).
\end{equation*}
From above equation, if $\pi$ is a horizontally homothetic map then $(\nabla \pi_{*})(\phi U,\xi)=0.$ Conversely, if $(\nabla \pi_{*})(\phi U,\xi)=0,$ we obtain
\begin{equation}
\phi U(\ln\lambda)\pi_{*}\xi+\xi(\ln\lambda)\pi_{*}\phi U=0.\label{e.q:3.12}
\end{equation}
Taking inner product in (\ref{e.q:3.12}) with $\pi_{*}\phi U$ and if we take into account $\pi$ is a conformal submersion, we write
\begin{equation*}
g_{M}(grad\ln\lambda,\phi U)g_{N}(\pi_{*}\xi,\pi_{*}\phi U)+g_{M}(grad\ln\lambda,\xi)g_{N}(\pi_{*}\phi U,\pi_{*}\phi U)=0.
\end{equation*}
Above equation implies that $\lambda$ is a constant on $\Gamma(\mu).$ On the other hand, taking inner product in (\ref{e.q:3.12}) with $\pi_{*}\xi$, we have
\begin{equation*}
g_{M}(grad\ln\lambda,\phi U)g_{N}(\pi_{*}\xi,\pi_{*}\xi)+g_{M}(grad\ln\lambda,\xi)g_{N}(\pi_{*}\phi U,\pi_{*}\xi)=0.
\end{equation*}
From above equation, it follows that $\lambda$ is a constant on $\Gamma(\phi ker\pi_{*}).$ Thus $\lambda$ is a constant on $\Gamma((ker\pi_{*})^\perp).$ Hence proof is complete.
\end{proof}
Here we present another result on conformal anti-invariant $\xi^\perp-$submersion to be totally geodesic.
\begin{theorem}
Let $\pi$ be a conformal anti-invariant $\xi^\perp-$submersion from a Sasakian manifold $(M,\phi,\xi,\eta,g_{M})$ to a Riemannian manifold $(N,g_{N})$. $\pi$ is a totally geodesic map if and only if
\begin{enumerate}
\item [(a)] $T_{U}\phi V=0$ and $\mathcal{H}\nabla_{U}\phi V\in\Gamma(\phi ker\pi_{*})$,
\item [(b)] $\pi$ is a horizontally homotetic map,
\item [(c)] $A_{Z}\phi V=0$ and $\mathcal{H}\nabla_{Z}\phi V\in\Gamma(\phi ker\pi)$
\end{enumerate}
for $X,Y,Z\in\Gamma((ker\pi_{*})^\perp)$ and $U,V\in\Gamma(ker\pi_{*})$.
\end{theorem}
\begin{proof}
For any $U,V\in\Gamma(ker\pi_{*})$, using (\ref{e.q:2.3}) and (\ref{nfixy}) we have
\begin{align*}
(\nabla \pi_{*})(U,V)&=\nabla^{\pi}_{U}\pi_{*}V+\pi_{*}(-\nabla_{U}V)\\
&=\pi_{*}(\phi\nabla_{U}\phi V-g_{M}(U,\phi V)\xi-\eta(V)X)\\
&=\pi_{*}(\phi\nabla_{U}\phi V).
\end{align*}
Then from (\ref{nvw}) and (\ref{nvx}) we arrive at
\begin{equation*}
(\nabla \pi_{*})(U,V)=\pi_{*}(\phi T_{U}\phi V+\mathcal{C}\mathcal{H}\nabla_{U}\phi V).
\end{equation*}
From above equation, $(\nabla \pi_{*})(U,V)=0$ if and only if
\begin{equation}
\pi_{*}(\phi T_{U}\phi V+\mathcal{C}\mathcal{H}\nabla_{U}\phi V)=0. \label{e.q:3.13}
\end{equation}
Since $\phi$ is non-singular, $T_{U}\phi V=0$ and $\mathcal{H}\nabla_{U}\phi V\in\Gamma(\phi ker\pi_{*}).$ On the other hand, from Lemma \ref{lem1} we derive
\begin{equation*}
(\nabla \pi_{*})(X,Y)=X(\ln\lambda)\pi_{*}Y+Y(\ln\lambda)\pi_{*}X-g_{M}(X,Y)\pi_{*}(grad\ln\lambda)
\end{equation*}
for any $X,Y\in\Gamma(\mu).$ It is obvious that if $\pi$ is a horizontally homotetic map, it follows that $(\nabla \pi_{*})(X,Y)=0.$ Conversely, if $(\nabla \pi_{*})(X,Y)=0,$ taking $Y=\phi X$ in above equation, we get
\begin{equation*}
X(\ln\lambda)\pi_{*}\phi X+\phi X(\ln\lambda)\pi_{*}X=0.
\end{equation*}
Taking inner product in (\ref{e.q:3.13}) with $\pi_{*}\phi X,$ we obtain
\begin{equation}
g_{M}(grad\ln\lambda,X)\lambda^{2}g_{M}(\phi X,\phi X)+g_{M}(grad\ln\lambda,\phi X)\lambda^{2}g_{M}(X,\phi X)=0. \label{e.q:3.14}
\end{equation}
From (\ref{e.q:3.14}), $\lambda$ is a constant on $\Gamma(\mu).$ On the other hand, for $U,V\in\Gamma(ker\pi_{*})$, from Lemma \ref{lem1} we have
\begin{equation*}
(\nabla \pi_{*})(\phi U,\phi V)=\phi U(\ln\lambda)\pi_{*}\phi V+\phi V(\ln\lambda)\pi_{*}\phi U-g_{M}(\phi U,\phi V)\pi_{*}(grad\ln\lambda).
\end{equation*}
Again if $\pi$ is a horizontally homothetic map, then $(\nabla \pi_{*})(\phi U,\phi V)=0.$ Conversely, if $(\nabla \pi_{*})(\phi U,\phi V)=0$, putting $U$
instead of $V$ in above equation, we derive
\begin{equation}
2\phi U(\ln\lambda)\pi_{*}(\phi U)-g_{M}(\phi U,\phi U)\pi_{*}(grad\ln\lambda)=0. \label{e.q:3.15}
\end{equation}
Taking inner product in (\ref{e.q:3.15}) with $\pi_{*}\phi U$ and since $\pi$ is a conformal submersion, we have
\begin{equation*}
g_{M}(\phi U,\phi U)\lambda^{2}g_{M}(grad\ln\lambda,\phi U)=0.
\end{equation*}
From above equation, $\lambda$ is a constant on $\Gamma(\phi ker\pi_{*}).$ Thus $\lambda$ is a constant on $\Gamma((ker\pi_{*})^\perp).$ Now, for $Z\in\Gamma(\mu)$ and $V\in\Gamma(ker\pi_{*}),$ from (\ref{e.q:2.3}) and (\ref{nfixy}) we get
\begin{align*}
(\nabla \pi_{*})(Z,V)=\pi_{*}(\phi\nabla_{Z}\phi V).
\end{align*}
Using (\ref{nxv}) and (\ref{nxy}) we have
\begin{equation*}
(\nabla \pi_{*})(Z,V)=\pi_{*}(\phi A_{Z}\phi V+\mathcal{C}\mathcal{H}\nabla_{Z}\phi V).
\end{equation*}
Thus $(\nabla \pi_{*})(Z,V)=0$ if and only if
\begin{equation*}
\pi_{*}(\phi A_{Z}\phi V+\mathcal{C}\mathcal{H}\nabla_{Z}\phi V)=0.
\end{equation*}
Since $\phi$ is non-singular, $A_Z{\phi V}=0$ and $\mathcal{H}\nabla_Z{\phi V}\in\Gamma(\phi ker\pi_*).$ Thus proof is complete.
\end{proof}

Finally, in this section, We investigate the necessary and sufficient conditions for such submersions to be harmonic.
\begin{theorem}\label{teo5}
Let $\pi:(M^{2(m+n)+1},\phi,\xi,\eta,g_M)\longrightarrow (N^{m+2n+1},g_{N})$ be a conformal anti-invariant $\xi^\perp-$submersion, where $(M,\phi,\xi,\eta,g_{M})$ is a Sasakian manifold and $(N,g_{N})$ is a Riemannian manifold. Then the tension field $\tau$ of $\pi$ is
\begin{align}
\tau(\pi)&=-m\pi_{*}(\mu^{ker\pi_{*}})+(1-m-2n)\pi_{*}(grad\ln\lambda)\label{e.q:3.16}
\end{align}
where $\mu^{ker\pi_{*}}$ is the mean curvature vector field of the distribution of $ker\pi_{*}$.
\end{theorem}
\begin{proof}
Let $\{e_{1},...,e_{m},\phi e_{1},...,\phi e_{m},\xi,\mu_1,...,\mu_n,\phi \mu_1,...,\phi \mu_n\}$ be orthonormal basis of $\Gamma(TM)$ such that $\{e_{1},...,e_{m}\}$ be orthonormal basis of $\Gamma(ker\pi_{*})$, $\{\phi e_{1},...,\phi e_{m}\}$ be orthonormal basis of $\Gamma(\phi ker\pi_{*})$ and $\{\xi,\mu_1,...,\mu_n,\phi \mu_1,...,\phi \mu_n\}$ be orthonormal basis of $\Gamma(\mu)$. Then the trace of second fundamental form (restriction to $ker\pi_{*}\times ker\pi_{*}$) is given by
$$trace^{ker\pi_{*}}\nabla \pi_{*}=\sum_{i=1}^{m}(\nabla \pi_{*})(e_{i},e_{i}).$$ Then using (\ref{nfixy}) we obtain
\begin{equation}
trace^{ker\pi_{*}}\nabla \pi_{*}=-m\pi_{*}(\mu^{ker\pi_{*}}).\label{e.q:3.17}
\end{equation}
In a similar way, we have
\begin{equation*}
trace^{(ker\pi_{*})^{\perp}}\nabla \pi_{*}=\sum_{i=1}^{m}(\nabla \pi_{*})(\phi e_{i},\phi e_{i})+\sum_{i=1}^{2n}(\nabla \pi_{*})(\mu_{i},\mu_{i})+(\nabla \pi_{*})(\xi,\xi).
\end{equation*}
Using Lemma \ref{lem1} we arrive at
\begin{align*}
trace^{(ker\pi_{*})^{\perp}}\nabla \pi_{*}&=\sum_{i=1}^{m}2g_{M}(grad\ln\lambda,\phi e_{i})\pi_{*}\phi e_{i}-m\pi_{*}(grad\ln\lambda)\\
&+\sum_{i=1}^{2n}2g_{M}(grad\ln\lambda,\mu_{i})\pi_{*}\mu_{i}-2n\pi_{*}(grad\ln\lambda)\\
&+2\xi(\ln\lambda)\pi_{*}\xi-\pi_{*}(grad\ln\lambda).
\end{align*}
Since $\{\frac{1}{\lambda(p)}\pi_{*p}(\phi e_i),\frac{1}{\lambda(p)}\pi_{*p}(\mu_h),\frac{1}{\lambda(p)}\pi_{*p}\xi\}_{p\in M,\ 1\leq i\leq m,\, 1\leq h\leq n}$ is an orthonormal basis of $T_{\pi(p)}N$ and $\pi$ is a conformal anti-invariant
$\xi^\perp$-submersion, we derive
\begin{align}
trace^{(ker\pi_{*})^{\perp}}\nabla \pi_{*}&=\sum_{i=1}^{m}2g_{N}(\pi_*grad\ln\lambda,\frac{1}{\lambda}\pi_*\phi e_{i})\frac{1}{\lambda}\pi_{*}\phi e_{i}-m\pi_{*}(grad\ln\lambda)\notag\\
&+\sum_{i=1}^{2n}2g_{N}(\pi_*grad\ln\lambda,\frac{1}{\lambda}\pi_*\mu_{i})\frac{1}{\lambda}\pi_{*}\mu_{i}-2n\pi_{*}(grad\ln\lambda)\notag\\
&+2g_N(\pi_*grad\ln\lambda,\frac{1}{\lambda}\pi_{*}\xi)\frac{1}{\lambda}\pi_{*}\xi-\pi_{*}(grad\ln\lambda)\notag\\
&=(1-m-2n)\pi_*(grad\ln\lambda)\label{e.q:3.18}
\end{align}
Then proof follows from (\ref{e.q:3.17}) and (\ref{e.q:3.18}).\\
\end{proof}
From Theorem \ref{teo5} we deduce that:
\begin{theorem}
Let $\pi:(M^{2(m+n)+1},\phi,\xi,\eta,g_M)\longrightarrow (N^{m+2n+1},g_{N})$ be a conformal anti-invariant $\xi^\perp-$submersion, where $(M,\phi,\xi,\eta,g_{M})$ is a Sasakian manifold and $(N,g_{N})$ is a Riemannian manifold. Then any two conditions below imply the third:
\begin{enumerate}
\item [(i)] $\pi$ is harmonic
\item [(ii)] The fibres are minimal
\item [(iii)] $\pi$ is a horizontally homothetic map.
\end{enumerate}
\end{theorem}

We also have the following result.
\begin{corollary}
Let $\pi:(M^{2(m+n)+1},\phi,\xi,\eta,g_M)\longrightarrow (N^{m+2n+1},g_{N})$ be a conformal anti-invariant $\xi^\perp-$submersion, where $(M,\phi,\xi,\eta,g_{M})$ is a Sasakian manifold and $(N,g_{N})$ is a Riemannian manifold. $\pi$ is harmonic if and only if the fibres are minimal.
\end{corollary}

\section{Decomposition theorems}
In this section, we obtain decomposition theorems by using the existence of
conformal anti-invariant $\xi^\perp-$submersions. First, we recall the following results
from \cite{P}. Let $g$ be a Riemannian metric tensor on the manifold $B=M\times N$
and assume that the canonical foliations $D_{M}$ and $D_{N}$ intersect
perpendiculary everywhere. Then $g$ is the metric tensor of

(i) a twisted product $M\times_{f}N$ if and only if $D_{M}$ is a
totally geodesic foliation and $D_{N}$ is a totally umbilic foliation,

(ii) a warped product $M\times_{f}N$ if and only if $D_{M}$ is a
totally geodesic foliation and $D_{N}$ is a spheric foliation, i.e., it is
umbilic and its mean curvature vector field is parallel. We note that
in this case, from \cite{P} we have
\begin{equation}
\nabla_XU=X(\ln f)U \label{e.q:4.1}
\end{equation}
for $X\in\Gamma(TM)$ and $U\in\Gamma(TN),$ where $\nabla$ is the Levi-Civita connection on $M\times N.$

(iii) a usual product of Riemannian manifolds if and only if $D_{M}$ and
$D_{N}$ are totally geodesic foliations.

Our first decomposition theorem for a conformal anti-invariant $\xi^\perp-$submersion comes
from Theorem \ref{teo2} and Theorem \ref{teo3} in terms of the second
fundamental forms of such submersions.

\begin{theorem}
Let $\pi:(M,\phi,\xi,\eta,g_{M})\longrightarrow (N,g_{N})$ is a conformal anti-invariant $\xi^\perp-$submersion, where $(M,\phi,\xi,\eta,g_{M})$ is a Sasakian manifold and $(N,g_{N})$ is a Riemannian manifold. Then $M$ is a locally product manifold if
\begin{align*}
-\frac{1}{\lambda^{2}}g_{N}(\nabla^{\pi}_{X}\pi_{*}\mathcal{C}Y,\pi_{*}\phi V)=g_{M}(A_{X}BY-\mathcal{C}Y(\ln\lambda)X+g_M(X,\mathcal{C}Y)\ln\lambda-\eta(Y)X,\phi V)
\end{align*}
and
\begin{align*}
-\frac{1}{\lambda^2}g_{N}(\nabla^{\pi}_{\phi W}\pi_{*}\phi V,\pi_{*}\phi\mathcal{C}X)&=g_{M}(\phi\mathcal{C}X(\ln\lambda)\phi V-T_{V}\mathcal{B}X, \phi V)+\eta(\nabla_{\phi W}V)\eta(\mathcal{C}X)
\end{align*}
for $X,Y\in\Gamma((ker\pi_{*})^\perp)$ and $U,V\in\Gamma(ker\pi_{*})$, where $M_{(\ker \pi_{\ast})^{\perp}}$ and $M_{(\ker \pi_{\ast})}$ are integral
manifolds of the distributions $(\ker \pi_{\ast})^{\perp}$ and $(\ker \pi_{\ast})$.
Conversely, if $M$ is a locally product manifold of the form
$M_{(\ker \pi_{\ast})^{\perp}}\times M_{(\ker \pi_{\ast})}$ then we have
\begin{align*}
\frac{1}{\lambda^{2}}g_{N}(\nabla^{\pi}_{X}\pi_{*}\mathcal{C}Y,\pi_{*}\phi V)=g_{M}(\mathcal{C}Y(\ln\lambda)X-g_M(X,\mathcal{C}Y)\ln\lambda+\eta(Y)X,\phi V)
\end{align*}
and
\begin{align*}
-\frac{1}{\lambda^2}g_{N}(\nabla^{\pi}_{\phi W}\pi_{*}\phi V,\pi_{*}\phi\mathcal{C}X)&=g_{M}(\phi\mathcal{C}X(\ln\lambda)\phi V,\phi V)+\eta(\nabla_{\phi W}V)\eta(\mathcal{C}X).
\end{align*}
\end{theorem}
From Corollary \ref{cor1} and Corollary \ref{cor2}, we have the following theorem.
\begin{theorem}
Let $\pi$ be a conformal anti-invariant $\xi^\perp-$submersion from a Sasakian manifold $(M,\phi,\xi,\eta,g_{M})$ onto a Riemannian manifold $(N,g_{N})$ with
$(ker\pi_{*})^{\perp}=\phi(ker\pi_{*})\oplus<\xi>$. Then $M$ is a locally product manifold if $A_{X}\mathcal{B}Y=\eta(Y)X$ and $T_{V}\phi W=0$ for $X,Y\in\Gamma((ker\pi_{*})^\perp)$ and $V,W\in\Gamma(ker\pi_{*})$.
\end{theorem}

Next we obtain a decomposition theorem which is related to the
notion of twisted product mani-fold. But we first recall the adjoint
map of a map. Let $\pi:(M_1,g_1)\rightarrow(M_2,g_2)$ be a map between
Riemannian manifolds $(M_1,g_1)$ and $(M_2,g_2).$ Then the adjoint
map $^{*}\pi_*$ of $\pi_*$ is characterized by
$g_1(x,^{*}\pi_{*p_1}y)=g_2(\pi_{*p_1}x,y)$ for $x\in T_{p_1}M_1,\, y\in
T_{\pi(p_1)}M_2$ and $p_1\in M_1.$ Considering $\pi^{h}_*$ at each
$p_1\in M_1$ as a linear transformation
$$\pi^{h}_{*p_1}:((\ker \pi_*)^\perp(p_1),g_{1_{p_1((\ker \pi_*)^\perp(p_1))}})\rightarrow (range \pi_*(p_2),g_{2_{p_2((range \pi_*)(p_2))}}),$$
we will denote the adjoint of $\pi^{h}_{*}$ by $^{*}\pi^{h}_{*p_1}.$ Let $^{*}\pi_{*p_1}$ be adjoint of
$\pi_{*p_1}:(T_{p_1}M_1,g_{1_{p_1}})\rightarrow (T_{p_2}M_2,g_{2_{p_2}}).$ Then the linear transformation
$$(^*\pi_{*p_1})^h:range \pi_*(p_2)\rightarrow (\ker \pi_*)^\perp(p_1)$$ defined by $(^{*}\pi_{*p_1})^{h}y= ^{*}\pi_{*p_1}y,$ where
$y\in\Gamma(range\pi_{*p_1}), p_2=\pi(p_1),$ is an isomorphism and $(\pi^{h}_{*p_1})^{-1}=(^{*}\pi_{*p_1})^h= ^{*}\pi^{h}_{*p_1}.$
\begin{theorem}
Let $\pi:(M,\phi,\xi,\eta,g_{M})\longrightarrow (N,g_{N})$ is a conformal anti-invariant $\xi^\perp-$submersion, where $(M,\phi,\xi,\eta,g_{M})$ is a Sasakian manifold and $(N,g_{N})$ is a Riemannian manifold. Then $M$ is a twisted product manifold of the from $M_{(ker\pi_{*})}\times_{\lambda} M_{(ker\pi_{*})^\perp}$ if and only if
\begin{align}
-\frac{1}{\lambda^2}g_{N}(\nabla^{\pi}_{\phi W}\pi_{*}\phi V,\pi_{*}\phi\mathcal{C}X)&=g_{M}(\phi\mathcal{C}X(\ln\lambda)\phi V-T_{V}\mathcal{B}X, \phi V)+\eta(\nabla_{\phi W}V)\eta(\mathcal{C}X)\label{e.q:4.2}
\end{align}
and
\begin{align}
g_{M}(X,Y)H&=-\mathcal{B}A_{X}\mathcal{B}Y+\mathcal{C}Y(\ln\lambda)\mathcal{B}X-\mathcal{B}\mathcal{H}(grad\ln\lambda)g_{M}(X,\mathcal{C}Y)\notag\\
&-\phi ^{*}\pi_*(\nabla^{\pi}_X\pi_{*}CY)+\eta(Y)\mathcal{B}X \label{e.q:4.3}
\end{align}
for $X,Y\in\Gamma((\ker \pi_{*})^{\perp})$ and $V,W\in\Gamma(\ker \pi_{*}),$
where $M_{(\ker \pi_{*})^{\perp}}$ and $M_{(\ker \pi_{*})}$ are integral
manifolds of the distributions $(\ker \pi_{*})^{\perp}$ and $(\ker \pi_{*})$
and $H$ is the mean curvature vector field of $M_{(\ker \pi_{*})^{\perp}}.$
\end{theorem}
\begin{proof}
Since $g_{M}(W,\xi)=0,$ using (\ref{e.q:2.3}) we have $g_{M}(\nabla_{V}W,\xi)=-g_{M}(W,\nabla_{V}\xi)=-g_{M}(W,\phi V)=0$
for $V,W\in\Gamma(ker\pi_{*})$ and $\xi\in\Gamma((ker\pi_{*})^\perp).$  Thus we have
\begin{align*}
g_{M}(\nabla_{V}W,X)&=g_{M}(\phi\nabla_{V}W,\phi X)+\eta(\nabla_{V}W)\eta(X)\\
&=g_{M}(\phi\nabla_{V}\phi W,\phi X)\\
&=g_{M}(\nabla_{V}\phi W,\phi X)-g_{M}((\nabla_{V}\phi)W,\phi X).
\end{align*}
Using (\ref{e.q:2.3}), (\ref{nvw}) and (\ref{e.q:3.2}) we have
\begin{equation*}
g_{M}(\nabla_{V}W,X)=g_{M}(T_{V}\phi W,\mathcal{B}X)+g_{M}(\mathcal{H}\nabla_{V}\phi W,\mathcal{C}X).
\end{equation*}
Since $\nabla$ is torsion free and $[V,\phi W]\in\Gamma(ker\pi_{*})$ we obtain
\begin{equation*}
g_{M}(\nabla_{V}W,X)=g_{M}(T_{V}\phi W,\mathcal{B}X)+g_{M}(\nabla_{\phi W}V,\mathcal{C}X).
\end{equation*}
Using (\ref{e.q:2.3}) and (\ref{nxy}) we have
\begin{align*}
g_{M}(\nabla_{V}W,X)&=g_{M}(T_{V}\phi W,\mathcal{B}X)+g_{M}(\phi\nabla_{\phi W}V,\phi\mathcal{C}X)+\eta(\nabla_{\phi W}V)\eta(\mathcal{C}X)\\
&=g_{M}(T_{V}\phi W,\mathcal{B}X)+g_{M}(\nabla_{\phi W}\phi V,\phi\mathcal{C}X)+\eta(\nabla_{\phi W}V)\eta(\mathcal{C}X)
\end{align*}
here we have used that $\mu$ is invariant. Using (\ref{nfixy}) and Lemma \ref{lem1} (i) and if we take into account
that $\pi$ is a conformal submersion, we obtain
\begin{align*}
g_{M}(\nabla_{U}V,X)&=g_{M}(T_{V}\phi W,\mathcal{B}X)+\frac{1}{\lambda^2}g_{M}(\mathcal{H}gradln\lambda,\phi W)g_{N}(\pi_{*}\phi V,\pi_{*}\phi\mathcal{C}X)\\
&-\frac{1}{\lambda^2}g_{M}(\mathcal{H}gradln\lambda,\phi V)g_{N}(\pi_{*}\phi W,\pi_{*}\phi\mathcal{C}X)\\
&+g_{M}(\phi W,\phi V)\frac{1}{\lambda^2}g_{N}(\pi_{*}(\mathcal{H}gradln\lambda),\pi_{*}\phi\mathcal{C}X)\\
&+\frac{1}{\lambda^2}g_{N}(\nabla_{\pi_{*}\phi W}\pi_{*}\phi V,\pi_{*}\phi\mathcal{C}X)+\eta(\nabla_{\phi W}V)\eta(\mathcal{C}X).
\end{align*}
Moreover, using Definition \ref{def} and (\ref{cywv}), we obtain
\begin{align*}
g_{M}(\nabla_{U}V,X)&=g_{M}(\phi\mathcal{C}X(\ln\lambda)\phi V-T_{V}\mathcal{B}X, \phi V)+\eta(\nabla_{\phi W}V)\eta(\mathcal{C}X)\\
&+\frac{1}{\lambda^2}g_{N}(\nabla^{\pi}_{\phi W}\pi_{*}\phi V,\pi_{*}\phi\mathcal{C}X).
\end{align*}
Thus it follows that $M_{(ker\pi_{*})}$ is totally geodesic if and only if the equation (\ref{e.q:4.2}) is satisfied. On the other hand, for $X,Y\in\Gamma((ker\pi_{*})^\perp)$ and $V\in\Gamma(ker\pi_{*})$, by using (\ref{e.q:2.2}), (\ref{nxv}), (\ref{nxy}), (\ref{e.q:3.1}), (\ref{e.q:3.2}) and (\ref{e.q:3.8}), have
\begin{align*}
g_{M}(\nabla_{X}Y,V)&=g_{M}(A_{X}BY,\phi V)+g_{M}(\nabla_{X}\mathcal{C}Y,\phi V)-\eta(Y)g_{M}(X,\phi V).
\end{align*}
Since $\pi$ is a conformal submersion, using (\ref{nfixy}) and Lemma (\ref{lem1}) we arrive at
\begin{align*}
g_{M}(\nabla_{X}Y,V)&=g_{M}(A_{X}BY,\phi V)-\frac{1}{\lambda^2}g_{M}(\mathcal{H}gradln\lambda,X)g_{N}(\pi_{*}\mathcal{C}Y,\pi_{*}\phi V)\\
&-\frac{1}{\lambda^2}g_{M}(\mathcal{H}gradln\lambda,\mathcal{C}Y)g_{N}(\pi_{*}X,\pi_{*}\phi V)\\
&+\frac{1}{\lambda^2}g_{M}(X,\mathcal{C}Y)g_{N}(\pi_{*}(\mathcal{H}gradln\lambda,\mathcal{C}Y),\pi_{*}\phi V)\\
&+\frac{1}{\lambda^2}g_{N}(\nabla_{\pi_{*}X}\pi_{*}\mathcal{C}Y,\pi_{*}\phi V)-\eta(Y)g_{M}(X,\phi V).
\end{align*}
Moreover, using Definiton \ref{def} and (\ref{cywv}) we obtain
\begin{align*}
g_{M}(\nabla_{X}Y,V)&=g_{M}(A_{X}BY-\mathcal{C}Y(\ln\lambda)X+g_M(X,\mathcal{C}Y)\ln\lambda-\eta(Y)X,\phi V)\\
&+\frac{1}{\lambda^2}g_{N}(\nabla_{\pi_{*}X}\pi_{*}\mathcal{C}Y,\pi_{*}\phi V).
\end{align*}
From above equation, $M_{(ker\pi_{*})^\perp}$ is totally umbilical if and only if the equation (\ref{e.q:4.3}) is satisfied.
\end{proof}
However, in the sequel, we show that the notion of conformal anti-invariant $\xi^\perp-$submersion puts some restrictions on the total space for locally warped product manifold.
\begin{theorem}
Let $\pi$ is a conformal anti-invariant $\xi^\perp-$submersion from a Sasakian manifold $(M,\phi,\xi,\eta,g_{M})$ to a Riemannian manifold $(N,g_{N})$. If $M$ is a locally warped product manifold of the from $M_{(ker\pi_{*})^\perp}\times_{\lambda} M_{(ker\pi_{*})}$, then either $\pi$ is a horizontally homothetic submersion or the fibres are one dimensional.
\end{theorem}
\begin{proof}
For $X\in\Gamma((ker\pi_{*})^\perp)$ and $V,W\in\Gamma(ker\pi_{*})$, $g_{M}(W,\xi)=0$ implies that from (\ref{e.q:2.3}), $g_{M}(\nabla_{V}W,\xi)=-g_{M}(W,\nabla_{V}\xi)=-g_{M}(W,\phi V)=0.$ Thus we have
\begin{align*}
g_{M}(\nabla_{V}W,X)&=g_{M}(\phi\nabla_{V}W,\phi X)+\eta(\nabla_{V}W)\eta(X)\\
&=g_{M}(\phi\nabla_{V}\phi W,\phi X)\\
&=g_{M}(\nabla_{V}\phi W,\phi X)-g_{M}((\nabla_{V}\phi)W,\phi X).
\end{align*}
Using (\ref{e.q:2.3}) and (\ref{e.q:4.1}), we get
\begin{align*}
g_{M}(\nabla_{V}W,X)&=g_{M}(\nabla_{V}\phi W,\phi X)-g_{M}(-g_{M}(V,W)\xi-\eta(W)V,\phi X)\\
-X(\ln\lambda)g_{M}(V,W)&=\phi W(\ln\lambda)g_{M}(V,\phi X).
\end{align*}
For $X\in\Gamma(\mu),$ we derive
\begin{equation*}
-X(\ln\lambda)g_{M}(V,W)=0.
\end{equation*}
From above equation, we conclude that $\lambda$ is a constant on $\Gamma(\mu).$ For $X=\phi V\in\Gamma(\phi ker\pi_{*})$ we obtain
\begin{equation}
\phi V(\ln\lambda)g_{M}(V,W)=\phi W(\ln\lambda)g_{M}(V,V).\label{e.q:4.4}
\end{equation}
Interchanging the roles of $W$ and $V$ in (\ref{e.q:4.4}) we arrive at
\begin{equation}
\phi V(\ln\lambda)g_{M}(V,W)=\phi V(\ln\lambda)g_{M}(W,W).\label{e.q:4.5}
\end{equation}
From (\ref{e.q:4.4}) and (\ref{e.q:4.5}) we get
\begin{equation}
\phi V(\ln\lambda)=\phi V(\ln\lambda)\frac{g_{M}(V,W)^{2}}{\parallel V \parallel^{2}\parallel W \parallel^{2} }.\label{e.q:4.6}
\end{equation}
From (\ref{e.q:4.6}), either $\lambda$ is a constant on $\Gamma(\phi ker\pi_{*})$ or $\Gamma(\phi ker\pi_{*})$ is 1-dimensional. Thus proof is complete.
\end{proof}

\end{document}